\documentclass{amsproc}
\usepackage{amssymb}
\usepackage{amsfonts}

\DeclareMathOperator{\R}{{\mathbb{R}}}
\DeclareMathOperator{\Id}{Id}
\DeclareMathOperator{\inte}{int}
\newcommand{\FTCco}{FTC\textsubscript{co}}

\newcommand{\GFTCco}{GFTC\textsubscript{co}}

\theoremstyle{plain}

\newtheorem{theorem}{Theorem}[section]

\newtheorem{claim}[theorem]{Claim}

\newtheorem{corollary}[theorem]{Corollary}

\newtheorem{definition}[theorem]{Definition}
\newtheorem{example}[theorem]{Example}

\newtheorem{lemma}[theorem]{Lemma}
\newtheorem{notation}[theorem]{Notation}

\newtheorem{proposition}[theorem]{Proposition}
\newtheorem{remark}[theorem]{Remark}

\numberwithin{equation}{section}

\begin{document}
\title[WSC is GFTC]{When the Weak Separation Condition implies the Generalized Finite Type Condition}
\author{Kathryn E. Hare}
\address{Dept. of Pure Mathematics\\
University of Waterloo\\
Waterloo, Canada N2L 3G1}
\email{kehare@uwaterloo.ca}
\thanks{KEH was supported by NSERC Grant 2016-03719. KGH was supported by
NSERC Grant 2019-03930. AR was supported by both these grants and the
University of Waterloo}
\author{Kevin G. Hare}
\address{Dept. of Pure Mathematics\\
University of Waterloo\\
Waterloo, Canada N2L 3G1}
\email{kghare@uwaterloo.ca}
\author{Alex Rutar}
\address{Dept. of Pure Mathematics\\
University of Waterloo\\
Waterloo, Canada N2L 3G1}
\email{arutar@uwaterloo.ca}
\thanks{This paper is in final form and no version of it will be submitted
for publication elsewhere.}

\begin{abstract}
We prove that an iterated function system of similarities on $\R$ that satisfies the weak separation condition and has an interval as its self-similar set is of generalized finite type.
It is unknown if the assumption that the self-similar set is an interval is necessary.
\end{abstract}

\maketitle

\section{Introduction}
Consider an iterated function system (IFS) of similarities $\{S_i\}_{i=1}^k$ on $\mathbb{R}^d$ and let $K$ be its associated self-similar set, the unique non-empty compact set satisfying  $K=\bigcup_{j=1}^{k}S_{j}(K)$. 
The dimensional properties of such sets are well understood if the IFS satisfies the open set condition (OSC), such as the IFS $\{x/3, x/3+2/3\}$ whose self-similar set is the classical middle-third Cantor set.
We refer the reader to \cite{Fa} and the many references cited there.
As many interesting IFS, including those associated with the much studied self-similar measures known as Bernoulli convolutions, do not have this property, Lau and Ngai in \cite{LN} introduced the weak separation condition (WSC), which permits limited types of overlap.
IFSs that satisfy this weaker property have also been intensively studied.
For instance, in \cite{Ze} Zerner proved many geometric and analytic equivalences, Feng and Lau in \cite{FL} obtained deep results about the multifractal analysis of associated self-similar measures, and Fraser et al in \cite{FHOR} showed that any self-similar set arising from an IFS satisfying this condition and not contained in a hyperplane is Ahlfors regular and hence its Hausdorff and Assouad dimensions coincide.
However, even basic concepts, such as explicitly computing the Hausdorff dimension of the self-similar set, can be challenging.

Many interesting examples of IFS satisfying the weak separation condition have a type of combinatorial overlap structure.
As a result, in-between separation notions, called the finite type condition (FTC) and the generalized finite type condition (GFTC), were introduced in \cite{NW} and \cite{LN2} respectively.
A class of examples of IFS of finite type that do not satisfy the OSC are the systems $\{\rho x, \rho x+1-\rho\}$ where $\rho>1/2$ is the inverse of a Pisot number.
These examples are particularly interesting as they are associated with the only Bernoullli convolutions known to be singular.
We refer the reader to \cite{Sh} and the survey paper \cite{Va} for more details on this much studied problem.
IFSs which satisfy the (generalized) finite type condition are much more tractable to study than those which merely satisfy the WSC.
For instance, simple formulas are known for the Hausdorff dimension of $K$ \cite{LN2, NW}, and the multifractal analysis of associated invariant measures is more precisely understood (compare, for example, \cite{F3,F9,HHS}~with~\cite{FL}).
Both finite type and generalized finite type are defined in terms of certain local geometric-combinatorial properties.

Much of the research that has been done on these properties applies to the case where the word `local' is interpreted to mean relative to the bounded invariant set which is the interior of the convex hull of the self-similar set $K$, when the interior is non-empty.
We call this the convex (generalized) finite type condition and write (G)FTC\textsubscript{co}.
The following relationships are well known, (c.f., \cite{DLN, LN2, LNR, Ng}):
\begin{equation}\label{Inc1}
    FTC_{co}\subsetneq GFTC_{co}\subseteq GFTC\subseteq WSC.
\end{equation}
Any IFS of finite type necessarily has logarithmically commensurate contraction factors (in fact, \GFTCco{} with commensurate contraction factors is \FTCco{} \cite{DLN}), so not all sets satisfying the open set condition are finite type.
However 
\begin{equation}\label{Inc2}
    OSC\subsetneq GFTC\subseteq WSC
\end{equation}
and if an IFS satisfies the open set condition with the open set being the interior of the convex hull of $K$, then it satisfies \GFTCco{}.
The question of whether the generalized finite type and weak separation conditions are equivalent was first raised by Lau and Ngai in \cite{LN2}.
They showed that the answer, in general, is negative for IFS on $\R^d$ for $d\geq 2$, but for all known examples, the interior of the convex hull of the self-similar set is empty.
In particular, we are not aware of any IFS on $\R$, with non-singleton attractor, which satisfies the weak separation condition but is not generalized finite type.

The main contribution of this paper, Theorem~\ref{t:wsc-wft}, is to partially answer this question.
We prove that if the IFS in $\R$ satisfies the weak separation condition and the self-similar set $K$ is an interval, then the IFS actually satisfies the (stronger) convex generalized finite type condition.

In \cite{F3}, Feng showed that the key features of the local geometry of an IFS in $\R$ with positive and equal contraction factors and satisfying \FTCco{} could be understood in terms of `neighbour sets' and that such IFS can have only finitely many of these.
This property has proven to be very fruitful in studying the multifractal analysis of self-similar measures of finite type, c.f. \cite{F3, F5, HHS}.
Here we generalize a slightly modified notion of a neighbour set to any IFS in $\R$, and say that any IFS which has only finitely many of these more general neighbour sets satisfies the finite neighbour condition.
The weak separation condition for IFS in $\R$ can be characterized in terms of the neighbour sets (Proposition~\ref{l:wsp-l}) and consequently any IFS in $\R$ satisfying the finite neighbour condition satisfies the WSC (Corollary~\ref{c:fnc-wsc}).

In Theorem~\ref{t:wft-fs} we prove that the finite neighbour condition and the convex generalized finite type condition coincide.
Our main result follows by proving that any IFS on $\R$ that has the interval $[0,1]$ as its self-similar set $K$ and satisfies the weak separation condition, has the finite neighbour condition and hence also the \GFTCco{}.
In particular, any such IFS with commensurate contraction factors has property \FTCco{}.
Our proof was inspired by Feng's recent work, \cite{F15}, that showed this under the additional restriction that the similarities have equal and positive contraction factors.

One reason for the interest in resolving the relationship between the weak separation condition and (generalized) finite type is that the geometric structure of the self-similar set and measures associated with an IFS satisfying the (generalized) finite type condition is substantially richer than what is directly implied by the weak separation condition.
For example, it is a consequence of our theorem that any IFS $\{S_{j}\}$ with $K=[0,1]$ and satisfying the WSC has the property that there is a constant $\varepsilon >0$ such that for any $\alpha >0$ and $u,v\in\{0,1\}$, either 
\begin{equation}\label{01sep}
    S_{\sigma}(u)=S_{\tau}(v)\text{ or }\left\vert S_{\sigma}(u)-S_{\tau}(v)\right\vert \geq \varepsilon \alpha
\end{equation}
whenever $S_{\sigma }$ and $S_{\tau }$ are compositions of the similarities $\{S_{j}\}$ with contraction factors approximately $\alpha$  (Corollary~\ref{c:wsc-equiv}).
Indeed, many of the combinatorial formalisms developed by Feng in \cite{F3} for positive, equicontractive IFS in $\R$ of finite type extend to iterated function systems in $\R$ satisfying the finite neighbour condition \cite{Ru}.
As a consequence, it will be shown there that under the assumption of $K=[0,1]$, many results for the multifractal analysis of self-similar measures on $\R$ of finite type, as established in \cite{F3, HHS}, can be extended to self-similar measures associated with IFS on $\R$ satisfying the weak separation condition.
In Example~\ref{ex:1}, we give an example of such an IFS that has non-commensurate contraction factors.
More examples, and more in-depth analysis, can be found in \cite{Ru}.

\section{Geometric structure of self-similar sets}
\subsection{Iterated function systems and separation conditions}
Our focus for the remainder of the paper will be on $\R$.
Thus by an iterated function system (IFS) $\mathcal{S} =\{S_{i}\}_{i=1}^{k}$ we mean a finite set of similarities 
\begin{equation}
    S_{i}(x)=r_{i}x+d_{i}:\mathbb{R}\rightarrow \mathbb{R}\text{ for each } i=1,2,\ldots,k,
\end{equation}%
with $0<\left\vert r_{i}\right\vert <1$ and $k\geq 2$.
The IFS is said to be (positive) equicontractive if all $r_{i}=r$ (and $r>0$).

A subset $V\subseteq\mathbb{R}$ is called invariant if $S_{j}(V)\subseteq V$ for all $j$.
Each IFS generates a unique non-empty, compact invariant set $K$ satisfying 
\begin{equation*}
    K=\bigcup_{j=1}^{k}S_{j}(K).
\end{equation*}
This set $K$ is known as the associated self-similar set.
We will assume $K$ is not a singleton.
By rescaling and translating the $d_{i}$, as needed, without loss of generality we may assume the convex hull of $K$ is $[0,1]$.

If we are also given probabilities $\{p_j\}_{j=1}^k$, meaning $p_j>0$ and $\sum_{j=1}^k p_j=1$,  then there is a unique probability measure $\mu$ satisfying
$$\mu(E)=\sum_{j=1}^k p_j \mu(S_j^{-1}(E))$$ for any Borel set $E \subset \R$.
This measure is referred to as a self-similar measure associated with the IFS and has as its support the self-similar set $K$.

For example, if we take the IFS $\{S_1(x)=x/3,S_2(x)=x/3+2/3\}$ and probabilities $p_1=1/2=p_2$, the self-similar set is the classical middle-third Cantor set and the self-similar measure is the uniform Cantor measure.
If we take the IFS $\{\rho x,\rho x+1-\rho\}$ with $0<\rho<1$ and the same equal probabilities, the self-similar measure is the Bernoulli convolution with parameter $\rho$.

\begin{definition}
The IFS $\mathcal{S}=\{S_{j}\}$ is said to satisfy the \textbf{open set condition } (OSC) if there is a non-empty bounded invariant open set $V$ such that $S_{i}(V)\cap S_{j}(V)$ is empty for all $i\neq j$.
\end{definition}
The IFS $\{x/3,x/3+2/3\}$ is such an example.

In contrast, the weak separation condition allows restricted overlap.
We introduce further notation to formally define this.
Let $\Sigma =\{1,\ldots ,k\}$ and $\Sigma^{\ast }$ denote the set of all the finite words on $\Sigma $.
Given $\sigma =(\sigma_{1},\ldots ,\sigma_{j})\in \Sigma^{\ast }$, we put 
\begin{equation*}
    \sigma^{-}=(\sigma_{1},\ldots ,\sigma_{j-1})\text{, }S_{\sigma }=S_{\sigma_{1}}\circ \cdots \circ S_{\sigma_{j}}\text{ and }r_{\sigma }=\prod_{i=1}^{j}r_{\sigma_{i}}.
\end{equation*}%
Given $\alpha >0,$ put
\begin{equation*}
    \Lambda_{\alpha }=\{\sigma \in \Sigma^{\ast }:|r_{\sigma }|<\alpha \leq |r_{\sigma^{-}}|\}.
\end{equation*}
We refer to $\sigma \in \Lambda_{\alpha }$ as the words of generation $\alpha $.
We remark that in the literature it is more common to see this defined by the rule $|r_{\sigma }|\leq \alpha <|r_{\sigma^{-}}|$.
The two choices are essentially equivalent, but this choice is more convenient for our purposes.

There are many equivalent ways to define the weak separation condition.
The following is item (5) on Zerner's list of equivalences in \cite[Thm. 1]{Ze}, and is the one of most use to us in this paper.

\begin{definition}\label{d:WSC}
    The IFS is said to satisfy the \textbf{weak separation condition} (WSC) if there is some $x_{0}\in \R$ and integer $N$ (or, equivalently, for all $x$ there is some $N$) such that for any $\alpha >0$ and finite word $\tau$, any closed ball with radius $\alpha$ contains no more than $N$ distinct points of the form $S_{\sigma }(S_{\tau }(x_{0}))$ for $\sigma \in \Lambda_{\alpha }$.
\end{definition}

It is well known that any IFS satisfying the open set condition satisfies the weak separation condition, but not conversely.
Examples of IFS that satisfy the WSC, but not the OSC, include the IFS $\{\rho x,\rho x+1-\rho \}$ where $\rho >1/2 $ is the inverse of a Pisot number, as well as the IFS $\left\{ x/d+j(d-1)/(md):j=0,...,m\right\}$ where $m\geq d\geq 2$ are integers.
With an appropriate choice of probabilities, the self-similar measure associated with the second IFS is the $m$-fold convolution of the uniform Cantor measure on the Cantor set with ratio of dissection $1/d$.
Convolutions of the Cantor measure are examples of invariant measures with interesting multifractal structure.
For instance, the $3$-fold convolution of the middle-third Cantor measure was the first example discovered to have an isolated point in its set of local dimensions \cite{HL}.

Both these families of examples actually satisfy a stronger separation condition known as finite type, a notion introduced by Ngai and Wang in \cite{NW}.
To explain this, and the more general notion of the generalized finite type condition introduced by Lau and Ngai in \cite{LN2}, we need further notation.

\begin{notation}
    For any set $V\subseteq \mathbb{R}$ let 
    \begin{equation}\label{ESNotation}
        \mathcal{E}_{\mathcal{S}}(V)=\bigcup\limits_{\alpha >0}\left\{S_{\sigma }^{-1}\circ S_{\tau }:\sigma ,\tau \in \Lambda_{\alpha},S_{\sigma }(V)\cap S_{\tau }(V)\neq \emptyset \right\}.
    \end{equation}
\end{notation}

It is immediate from the definition that the IFS $\mathcal{S}=\{S_{j}\}$ satisfies the OSC with the bounded, invariant open set $V$ precisely when $\mathcal{E}_{\mathcal{S}}(V)$ consists of simply the identity map.

\begin{definition}
    (i) The IFS $\mathcal{S}=\{S_{j}\}$ is said to be of \textbf{generalized finite type}, or satisfy the \textbf{generalized finite type condition} (GFTC), if $\mathcal{E}_{\mathcal{S}}(V)$ is finite for some non-empty bounded invariant open set $V$.

    (ii) The IFS $\mathcal{S}=\{S_{j}\}$ is said to be of \textbf{finite type}, or satisfy the \textbf{finite type condition} (FTC) if, in addition, the contraction factors of the $S_{j}$ are logarithmically commensurate.
\end{definition}

By logarithmically commensurate we mean that the contraction factors $\{r_{i}\}$ have the property that for all $i,j$, $\log \left\vert r_{i}\right\vert /\log \left\vert r_{j}\right\vert \in \mathbb{Q}$.
We remark that the definitions given above were not the original definitions, but were proven to be equivalent by Deng et al in \cite[Thm. 4.1]{DLN}.

 Lau and Ngai in \cite{LN2} show that the IFS $\{\rho x, rx+\rho (1-r), rx+1-r\}$ for $0<\rho, r<1$ and $\rho+2r-\rho r \le 1$ is of generalized finite type, but not, in general, of finite type or satisfy the OSC. 

One of the main accomplishments of Feng in \cite{F3} was to show that positive, equicontractive IFS of finite type with the invariant open set $V=(0,1)$ have a special geometric structure which is very useful in studying both the self-similar set and the multifractal analysis of associated self-similar measures. 

In this paper, we will see that a similar geometric structure also holds for IFS of  generalized finite type when $V=(0,1)$ and hence we give this special case a name.

\begin{definition}
    The IFS is said to satisfy the \textbf{convex (generalized) finite type condition} 
((G)FTC$_{co}$) if it is of (generalized) finite type with the non-empty bounded invariant open set being the interior of the convex hull of $K$.
\end{definition}
\begin{remark}
We prefer to express this in terms of the convex hull of $K,$ rather than $(0,1)$, as the convex hull will generalize to IFS defined on $\R^n$ for $n>1$.
Note that we implicitly require that the interior of the convex hull is non-empty or, equivalently, that the attractor $K$ is not contained in a hyperplane.
The interior of the convex hull of $K$ is always an invariant set.
\end{remark}

Clearly \FTCco{} $\subseteq$ \GFTCco{}, and it is known that GFTC $\subseteq$ WSC, \cite{LN2}.
We are not aware of any example in $\R^d$  of an IFS where the interior of the convex hull of $K$ is non-empty and which satisfies the WSC, but not the \GFTCco{}.

\subsection{Neighbour sets}

The notions of net intervals and neighbour sets, introduced in \cite{F3} and \cite{HHS}, have proven to be very fruitful in the study of IFS satisfying the finite type condition.
Here we extend these notions to an arbitrary IFS in $\R$ where the attractor $K$ is not a singleton or, equivalently, the interior of the convex hull of $K$ is non-empty.

Let $h_{1},\ldots ,h_{s(\alpha )}$ be the collection of distinct elements of the set $\{S_{\sigma }(0),S_{\sigma }(1):\sigma \in \Lambda_{\alpha }\}$ listed in strictly ascending order and let 
\begin{equation*}
    \mathcal{F}_{\alpha }=\{[h_{j},h_{j+1}]:1\leq j<s(\alpha )\text{ and }(h_{j},h_{j+1})\cap K\neq \emptyset \}.
\end{equation*}
Elements of $\mathcal{F}_{\alpha }$ are called \textbf{net intervals of generation }$\alpha $.
For convenience, we write $\mathcal{F} =\bigcup_{\alpha >0}\mathcal{F}_{\alpha }$ to denote the set of all possible net intervals.

Suppose $\Delta \in \mathcal{F}$.
We denote by $T_{\Delta }$ the unique contraction $T_{\Delta }(x)=rx+a$ with $r>0$ such that 
\begin{equation*}
    T_{\Delta }([0,1])=\Delta.
\end{equation*}%
Of course, $r=m(\Delta)$ where $m$ denotes the normalized Lebesgue measure and $a$ is the left endpoint of $\Delta$.

\begin{definition}\label{d:nb}
    We will say that a similarity $T(x)=Lx+a$ is a \textbf{neighbour} of $\Delta \in \mathcal{F}_{\alpha }$ if there exists some $\sigma \in \Lambda_{\alpha }$ such that $S_{\sigma }([0,1])\supseteq \Delta $ and $T=T_{\Delta }^{-1}\circ S_{\sigma }$.
    In this case, we also say that $S_{\sigma }$ \textbf{generates} the neighbour $T$.
    The \textbf{neighbour set} of $\Delta $ is the maximal set 
    \begin{equation*}
        V_{\alpha }(\Delta )=\{T_{1},\ldots ,T_{m}\}
    \end{equation*}
    where each $T_{i}=T_{\Delta }^{-1}\circ S_{\sigma_{i}}$ is a distinct neighbour of $\Delta $.
    When the generation of $\Delta$ is implicit, we will simply write $V(\Delta )$.
\end{definition}

Fix a net interval $\Delta =[a_{0},b_{0}]\in \mathcal{F}_{\alpha }$ and $\sigma \in \Lambda_{\alpha }$ with $S_{\sigma }([0,1])\supseteq \Delta $.
Put $L=r_{\sigma }/m(\Delta )$ and $a=(S_{\sigma }(0)-a_{0})/m(\Delta )$.
Then the neighbour $T$ generated by $S_{\sigma }$ is given by 
\begin{equation*}
    T=T_{\Delta }^{-1}\circ S_{\sigma }(x)=\frac{S_{\sigma }(x)-a_{0}}{b_{0}-a_{0}}=\frac{r_{\sigma }x+S_{\sigma }(0)-a_{0}}{m(\Delta )}=Lx+a.
\end{equation*}
The neighbours $T$ of $\Delta $ are clearly in one-to-one correspondence with the pairs $(a,L)$.
Thus this definition of a neighbour is a slightly modified version of the one defined in \cite{HHS}, where instead of normalizing by a value $\alpha =r_{\min }^{n}$ we normalize by $m(\Delta)$.

\subsection{Characterizing the weak separation condition by neighbour sets}

A characterization of the weak separation condition can be easily expressed in terms of these neighbour sets.
This forward implication of this proposition captures the intuition in $\R^d$ that, under the weak separation condition (when the attractor of the IFS is not contained in a hyperplane), the choice of $N$ in Definition~\ref{d:WSC} can be made independent of the initial choice of some fixed $x_0$.
This observation is used, for example, in the proof of \cite[Thm. 2.1]{FHOR} that IFSs which satisfy the weak separation condition are Ahlfors regular.

Recall that we assume that our IFS is in $\R$ and that the convex hull of the invariant compact set is $[0,1]$.
We write $\#X$ for the cardinality of the set $X$.
\begin{proposition}\label{l:wsp-l}
    An IFS has the weak separation condition if and only if $\sup_{\Delta \in \mathcal{F}}\#V(\Delta )<\infty$.
\end{proposition}
\begin{proof}
    Suppose the IFS $\mathcal{S}=\{S_{j}\}$ has the weak separation condition.
    Find the bounds $N_{1},N_{2}$ from Definition~\ref{d:WSC} (of the WSC) with the points $x_{0}=0,1$ and $\tau =\Id$.
    Let $[a_{0},b_{0}]=\Delta \in \mathcal{F}_{\alpha }$ be an arbitrary net interval and define 
    \begin{equation*}
        E=\{S_{\sigma }(0),S_{\sigma }(1):T_{\Delta }^{-1}\circ S_{\sigma }\in V(\Delta )\}.
    \end{equation*}
    Note that $\#V(\Delta )\leq \#E(\#E-1)$ since to any neighbour of $\Delta$ there must correspond two distinct points in $E$.
    Thus it suffices to show that $\#E$ is bounded.

    Consider the closed sets $I_{0}=[b_{0},b_{0}+\alpha ]$ and $I_{1}=[a_{0}-\alpha ,a_{0}]$, set $I=I_{0}\cup I_{1}$, and note that $E\subseteq I$.
    By the definition of the weak separation condition, $I$ contains at most $2N_{1}$ distinct points of the form $S_{\sigma }(0)$ for $\sigma \in \Lambda_{\alpha }$, and at most $2N_{2}$ distinct points of the form $S_{\sigma }(1)$ for $\sigma \in \Lambda_{\alpha }$.
    Thus $\#E\leq 2(N_{1}+N_{2})$.

    Conversely, suppose $\sup_{\Delta \in \mathcal{F}}\#V(\Delta )=M<\infty $.
    Fix an arbitrary generation $\alpha $ and any closed ball $I$ with radius $\alpha $.
    Fix $x_{0}\in \lbrack 0,1]$. Let $J$ be a closed ball with radius $2\alpha $ and the same center as $I$.
    Assume $x=S_{\sigma }(x_{0})\in I\subseteq J$ for some $\sigma \in \Lambda_{\alpha }$.

    If $I$ contains $N$ distinct points of the form $S_{\sigma }(x_{0})$ for $\sigma \in \Lambda_{\alpha }$, then $\#Y\geq N$ where 
    \begin{equation*}
        Y=\{S_{\sigma }:S_{\sigma }(x_{0})\in I,\sigma \in \Lambda_{\alpha}\}.
    \end{equation*}%
    Thus $\sum_{f\in Y}m(f([0,1])\geq N\alpha r_{\min }$.

    On the other hand, since $\sup_{\Delta \in \mathcal{F}}\#V(\Delta )=M$, any point $x\in J$ can be contained in the interior of at most $M$ intervals $f([0,1])$ with $f\in Y$.
    Hence $\sum_{f\in Y}m(f([0,1])\leq Mm(J)=M4\alpha$.
    Combining these two inequalities yields 
    \begin{equation*}
        N\alpha r_{\min }\leq M4\alpha ,\text{ or equivalently, }N\leq \frac{4M}{r_{\min }}.
    \end{equation*}%
    This gives a uniform bound for $N$ and hence the IFS satisfies the weak separation condition.
\end{proof}

\section{Finite neighbour condition\label{S:WFTC}}

IFSs that have the finite type property admit only finitely many neighbour sets, \cite{F3}.
(Strictly speaking, neighbour sets were only defined for the generations $r_{\min }^{k}$ for $k\in \mathbb{N}$, but, as shown in \cite[Prop. 5.4]{DLN}, this makes no difference.)
This was a crucial feature in carrying out the multifractal analysis of self-similar measures associated with IFS of finite type in \cite{F3, F5, HHS}.
Inspired by this, we make the following definition.
\begin{definition}
    We say that an IFS satisfies the \textbf{finite neighbour condition } if there are only finitely many neighbour sets.
\end{definition}

We emphasize that there is no requirement here that the contractions be logarithmically commensurate, as is implicitly required with finite type property.
\begin{example}\label{ex:1}
    Consider the IFS given by the four maps
    \begin{align*}
        S_1(x) &= \frac{1}{3}x & S_2(x) &= \frac{1}{4}x+\frac{1}{4}\\
        S_3(x) &= \frac{1}{4}x+\frac{1}{2} & S_4(x) &= \frac{1}{4}x+\frac{3}{4}
    \end{align*}
    which has invariant set $K=[0,1]$.
    This example is very similar to \cite[Example 2.8.]{LN2}.
    Following similar methods, it is straightforward to show that there are five possible neighbour sets given by
    \begin{align*}
        v_1 &= \{x\mapsto x\} & v_2 &= \{x\mapsto\frac{4}{3}x\}\\
        v_3 &= \{x\mapsto 3x,x\mapsto 4x-3\} &v_4 &= \{x\mapsto \frac{3}{2}x-\frac{1}{2}\}\\
        v_5 &= \{x\mapsto x,x\mapsto 3x\}
    \end{align*}
    so that this IFS satisfies the finite neighbour condition, and hence the weak separation condition.
\end{example}
Many other examples are given in \cite{Ru}.

If there are only finitely many neighbour sets, then $\sup_{\Delta \in \mathcal{F}}\#V(\Delta )<\infty$.
Consequently, Proposition~\ref{l:wsp-l} immediately gives

\begin{corollary}\label{c:fnc-wsc}
Any IFS satisfying the finite neighbour condition has the weak separation condition.
\end{corollary}

Next, we will prove that finite neighbour condition and \GFTCco{} coincide.
We begin with a general construction.
Let $\Gamma $ be any finite set of similarities on $\mathbb{R}$ and define 
\begin{equation*}
    \mathcal{N}(\Gamma ):=\left\{ x\mapsto \frac{f_{3}(x)-f_{2}(v_{2})}{f_{1}(v_{1})-f_{2}(v_{2})}:v_{i}\in \{0,1\},f_{i}\in \Gamma ,f_{1}(v_{1})-f_{2}(v_{2})\neq 0\right\}.
\end{equation*}
Clearly $\mathcal{N}(\Gamma)$ is a finite set since $\Gamma $ is finite.

The motivation behind this construction is the following: suppose we are given an arbitrary net interval $\Delta =[f(u),g(v)]$, where $f,g\in \Gamma$ and $u,v\in \{0,1\}$.
Suppose that $T$ is a neighbour of $\Delta$ generated by $S_{\sigma }\in \Gamma $.
Then 
\begin{equation*}
    T(x)=T_{\Delta }^{-1}\circ S_{\sigma }(x)=\frac{S_{\sigma }(x)-f(u)}{g(v)-f(u)}\in \mathcal{N}(\Gamma ).
\end{equation*}
Notice that if $S$ is any similarity and we set $S(\Gamma )=\{S\circ f:f\in \Gamma \},$ then one can easily check that $\mathcal{N}(S(\Gamma ))=\mathcal{N}(\Gamma)$.

\begin{theorem}\label{t:wft-fs}
    The IFS $\mathcal{S}$ satisfies the finite neighbour condition if and only if it satisfies the convex generalized finite type condition.
\end{theorem}

\begin{proof}
    As in (\ref{ESNotation}), set
    \begin{equation*}
        \mathcal{E}:=\mathcal{E}_{\mathcal{S}}((0,1))=\bigcup\limits_{\alpha>0}\{S_{\sigma }^{-1}\circ S_{\tau }:\sigma ,\tau \in \Lambda_{\alpha},S_{\sigma }((0,1))\cap S_{\tau }((0,1))\neq \emptyset \}.
    \end{equation*}
    According to the definitions we have given of \GFTCco{} and the finite neighbour condition, the theorem is equivalent to the statement that $\mathcal{S}$ has finitely many neighbour sets if and only if $\mathcal{E}$ is finite.

    First, suppose $\mathcal{S}$ has only finitely many neighbour sets.
    Let $\sigma ,\tau \in \Lambda_{\alpha }$ be arbitrary and suppose $I=S_{\sigma}((0,1))\cap S_{\tau }((0,1))\neq \emptyset $.
    Then there exists some net interval $\Delta \in \mathcal{F}_{\alpha}$ contained in $I$, so that $S_{\sigma }$ and $S_{\tau }$ generate neighbours of $\Delta $.
    In particular, $T_{\Delta }^{-1}\circ S_{\sigma }$ and $T_{\Delta}^{-1}\circ S_{\tau }$ must be two of the finitely many neighbours.
    Hence 
    \begin{equation*}
        S_{\sigma }^{-1}\circ S_{\tau }=S_{\sigma }^{-1}\circ T_{\Delta }\circ T_{\Delta }^{-1}\circ S_{\tau }=(T_{\Delta }^{-1}\circ S_{\sigma})^{-1}\circ (T_{\Delta }^{-1}\circ S_{\tau })
    \end{equation*}
    can only take finitely many values, so $\mathcal{E}$ is a finite set.

    Conversely, suppose $\mathcal{E}$ is a finite set.
    Let $\Delta =[a,b]\in \mathcal{F}_{\alpha }$ be an arbitrary net interval.
    Let $S_{\sigma_{0}}$ generate a neighbour $T$ of $\Delta $ and $\sigma_{1},\sigma_{2}$ be such that $a\in \{S_{\sigma_{1}}(0),S_{\sigma_{1}}(1)\}$ and $b\in \{S_{\sigma_{2}}(0),S_{\sigma_{2}}(1)\}$.
    It can always be arranged for $S_{\sigma_{0}}((0,1))$ to intersect non-trivially with both $S_{\sigma_{1}}((0,1))$ and $S_{\sigma_{2}}((0,1))$.
    But then, by the invariance of $\mathcal{N}$ under composition by similarities, we have 
    \begin{equation*}
        T\in \mathcal{N}(\{S_{\sigma_{0}},S_{\sigma_{1}},S_{\sigma_{2}}\})=\mathcal{N}(\{S_{\sigma_{0}}^{-1}\circ S_{\sigma_{0}},S_{\sigma_{0}}^{-1}\circ S_{\sigma_{1}},S_{\sigma_{0}}^{-1}\circ S_{\sigma_{2}}\})\subseteq \mathcal{N}(\mathcal{E}).
    \end{equation*}
    Since $\mathcal{E}$ is a finite set, there are only finitely many neighbour sets.
\end{proof}

\begin{corollary}\label{WFTC=GFTC}
    An equicontractive IFS $\mathcal{S}$ with contraction factor $\rho >0$ has the finite neighbour condition if and only if there is a finite set $\Gamma $ such that for each $n\in \mathbb{N}$ and $\sigma ,\tau \in \Lambda_{n}$ we have either
    \begin{equation}\label{FTCFeng}
        \rho^{-n}\left\vert S_{\sigma }(0)-S_{\tau }(0)\right\vert \geq 1\text{ or } \rho^{-n}\left\vert S_{\sigma }(0)-S_{\tau }(0)\right\vert \in \Gamma .
    \end{equation}
\end{corollary}
\begin{proof}
    If $\sigma ,\tau \in \Lambda_{\rho^{n}}$, then $S_{\sigma }((0,1))\cap S_{\tau }((0,1))=\emptyset $ if and only if we have $\rho^{-n}\left\vert S_{\sigma }(0)-S_{\tau }(0)\right\vert \geq 1$.
    Note that if $S_{\sigma}^{-1}\circ S_{\tau }(x)=x+d_{1}$ (we recall that the contraction factors of $S_{\sigma }$ and $S_{\tau }$ are necessarily equal) and $\rho^{-n}\left\vert S_{\sigma }(0)-S_{\tau }(0)\right\vert =d_{2}\in \Gamma$, then $\left\vert d_{1}\right\vert =d_{2}$.
    Thus
    \begin{equation*}
        \Gamma \subseteq \{\left\vert d\right\vert :rx+d\in \mathcal{E}_{\mathcal{S}}((0,1))\}\subseteq \{\pm d:d\in \Gamma \}.
    \end{equation*}

    Hence $\mathcal{E}_{\mathcal{S}}((0,1))$ is finite (equivalently, $\mathcal{S}$ has the finite neighbour condition) if and only if $\Gamma $ is finite.
    And, of course, the properties \GFTCco{} and finite neighbour condition coincide.
\end{proof}

\begin{remark}
    More generally, it is immediate from the theorem that the finite neighbour condition is equivalent to \FTCco{} if the contraction factors are commensurate.
\end{remark}
Any IFS that has the \FTCco{} property also has the property that there exists some $c>0$ such that for any $\alpha >0$ and $\Delta \in \mathcal{F}_{\alpha}$, it is the case that $m(\Delta )\geq c\alpha $ (\cite{F3, HHS}).
Since the images of $0$ and $1$ under the maps $S_{\sigma }$ for $\sigma \in $ $\Lambda_{\alpha }$ are the endpoints of the net intervals in $\mathcal{F}_{\alpha }$, in the case that $K=[0,1]$ this property equivalent to saying that there exists some $c>0$ such that for any $0<\alpha \leq 1$, words $\sigma ,\tau \in \Lambda_{\alpha }$ and $z,w\in \{0,1\}$, either 
\begin{equation}\label{BSP}
    S_{\sigma }(z)=S_{\tau }(w)\text{ or }\left\vert S_{\sigma }(z)-S_{\tau}(w)\right\vert \geq c\alpha.
\end{equation}
In fact, (\ref{BSP}) holds, even without the assumption that $K=[0,1]$, for IFS satisfying the finite neighbour condition.

\begin{theorem}
    Suppose $\mathcal{S}$ has the finite neighbour condition.
    Then there exists some $c>0$ such that for any $0<\alpha \leq 1$, $\sigma ,\tau \in \Lambda_{\alpha }$, and $u,v\in \{0,1\}$, 
    \begin{equation*}
        \text{either}\quad S_{\sigma }(u)=S_{\tau }(v)\quad \text{or}\quad |S_{\sigma }(u)-S_{\tau }(v)|\geq c\alpha.
    \end{equation*}
\end{theorem}

\begin{proof}
    Again, let $\mathcal{E=E}_{\mathcal{S}}((0,1))$.
    This is a finite set since $\mathcal{S}$ has the finite neighbour condition, equivalently \GFTCco{}.
    Let $\mathcal{G}$ denote the finite set 
    \begin{equation*}
        \mathcal{G}=\{g^{-1}\circ f\circ h:f\in \mathcal{E};g,h\in \{\Id,S_{1},\ldots ,S_{k}\}\}
    \end{equation*}%
    Likewise, set $\mathcal{V}=\bigcup_{\Delta \in \mathcal{F}}V(\Delta )$ to denote the set of all neighbours, so that $\mathcal{V}$ is a finite set.
    Then put 
    \begin{align*}
        c_{1}& :=\min \{1/|L|:\{x\mapsto Lx+a\}\in \mathcal{V}\}, \\
        c_{2}& :=\min \{|f(u)-g(v)|:u,v\in \{0,1\},f,g\in \mathcal{G},f(u)\neq g(v)\}.
    \end{align*}
    and let
    \begin{equation*}
        c:=r_{\min }\cdot \min \{c_{1},c_{2}\}.
    \end{equation*}%
    We will see that $c$ satisfies the requirements.

    Let $\sigma ,\tau \in \Lambda_{\alpha }$ and assume $S_{\sigma }(u)<S_{\tau}(v)$.
    Put $I=[S_{\sigma }(u),S_{\tau }(v)]$.

    If $I$ contains a net interval $\Delta \in \mathcal{F}_{\alpha }$, then $\Delta $ has some neighbour generated by a word $\omega$.
    In particular, $m(\Delta )/|r_{\omega }|\geq c_{1}$ by definition of a neighbour, so that $m(I)\geq m(\Delta )\geq \alpha r_{\min }c_{1}\geq \alpha c$.

    Otherwise, there is no net interval contained in $I$, equivalently, $\inte I\cap K=\emptyset$.
    Without loss of generality, we may assume that $\alpha $ is maximal with the property that $S_{\sigma }(u)$ and $S_{\tau}(v) $ are both endpoints of generation $\alpha $.
    Fix $\alpha^{\prime}=\min \{|r_{\sigma^{-}}|,|r_{\tau^{-}}|\}$ and obtain $\sigma^{\prime},\tau^{\prime }\in \Lambda_{\alpha^{\prime }}$, prefixes of $\sigma$ and $\tau$ respectively.
    Note that $(\sigma^{\prime },\tau^{\prime })$ is one of $(\sigma^{-}, \Id)$, $(\Id,\tau^{-})$ or $(\sigma^{-},\tau^{-})$.

    Since $\inte I\cap K=\emptyset $, $I$ contains no endpoints of generation $\alpha^{\prime }$, hence the maximality of $\alpha $ implies (without loss of generality) that we have $\sigma^{\prime }=\sigma^{-}$ and $I\subseteq S_{\sigma^{\prime }}([0,1])$.
    If $S_{\sigma^{\prime}}((0,1))\cap S_{\tau^{\prime }}((0,1))$ is not empty, then $S_{\sigma^{\prime }}^{-1}\circ S_{\tau^{\prime }}\in \mathcal{E}$ and hence $S_{\sigma }^{-1}\circ S_{\tau }\in \mathcal{G}$.
    Thus 
    \begin{equation}\label{BSC}
        |S_{\sigma }(u)-S_{\tau }(v)|=r_{\sigma }|u-S_{\sigma }^{-1}\circ S_{\tau}(v)|\geq c_{1}r_{\min }\alpha \geq c\alpha.
    \end{equation}
    If, instead, $S_{\sigma^{\prime }}((0,1))\cap S_{\tau^{\prime }}((0,1))$ is empty, then since $\inte I\cap K=\emptyset $ we have that $S_{\tau}(v)=S_{\sigma^{\prime }}(z)$ for some $z\in \{0,1\}$.
    Now apply the inequality (\ref{BSC}) with $S_{\sigma^{\prime }}(z)$ in place of $S_{\tau}(v)$, noting that $S_{\sigma }^{-1}\circ S_{\sigma^{\prime }}=S_{j}^{-1}$ for some $j=1,...,k$ and thus belongs to $\mathcal{G}$.
    Again, we deduce that $|S_{\sigma }(u)-S_{\tau }(v)|\geq c\alpha $, as required.
\end{proof}

\begin{remark}
This gives another proof that the finite neighbour condition implies the
weak separation condition.
\end{remark}

\section{Equivalence of the weak separation condition and finite neighbour condition}

In this section we will prove our main result, that the weak separation condition coincides with finite neighbour condition if the self-similar set is the full interval $[0,1]$.
Our technique is motivated by Feng's proof, \cite{F15}, that under the additional assumption of equal, positive contractions such IFS are finite type.
First, we prove a technical result.

\begin{lemma} \label{l:delsep}
    Fix some $\delta>0$.
    There exists some constant $C=C(\delta )>0$ such that for any $\alpha >0$ and $\sigma ,\tau \in \Lambda_{\alpha }$ with $m(S_{\sigma }([0,1])\cap S_{\tau }([0,1]))\geq \delta \alpha $, there is some word $\phi $ with $|r_{\phi }|\geq C$ and a choice of $\psi \in \{\sigma ,\tau \}$ such that $r_{\psi \phi }>0$ and 
    \begin{equation*}
        S_{\psi\phi }([0,1])\subseteq S_{\sigma }([0,1])\cap S_{\tau }([0,1]).
    \end{equation*}
\end{lemma}

\begin{proof}
    Say $S_{\sigma }([0,1])\cap S_{\tau }([0,1])=[c,d]$ with $d-c\geq \delta\alpha $.
    Without loss of generality, we may assume that $d=S_{\sigma }(v)$ for some $v\in \{0,1\}$, and we put $\psi =\sigma $.
    Let $i_{0}$ be an index with $0\in S_{i_{0}}([0,1])$ and $i_{1}$ be chosen so that $1\in S_{i_{1}}([0,1])$.
    Set $C=C(\delta ):=\delta r_{\min }^{2}$.

    We recursively construct $\phi $ as follows:

    \begin{itemize}
        \item If $r_{\psi }>0$, set $\phi_{1}=(i_{1})$, while if $r_{\psi }<0$, take $\phi_{1}=(i_{0})$.
            This choice of $\phi_{1}$ ensures that $S_{\psi\phi_{1}}([0,1])=[c_{1},S_{\psi }(v)]$ for some $c_{1}<S_{\psi }(v)$.

        \item Given $\phi_{n},$ a word of length $n$ such that $S_{\psi }(v)\in S_{\psi \phi_{n}}([0,1])$, take $\phi_{n+1}=\phi_{n}i_{1}$ if $r_{\psi \phi_{n}}>0$, and take $\phi_{n+1}=\phi_{n}i_{0}$ if $r_{\psi \phi_{n}}<0$.
            Again, $S_{\psi \phi_{n+1}}([0,1])=[c_{n+1},S_{\psi }(v)]$.
    \end{itemize}

    Let $N$ be minimal so that $S_{\psi }(v)-c_{N}\leq \delta \alpha $ and thus $S_{\psi \phi_{N}}([0,1])\subseteq \lbrack c,d]$.
    Note that $S_{\psi }(v)-c_{N}=|r_{\phi_{N}}r_{\psi }|$, so by the minimality of $N$, $|r_{\psi }r_{\phi_{N}}|\geq \delta \alpha r_{\min }$.
    Since $|r_{\psi }|\leq \alpha $, that ensures $|r_{\phi_{N}}|\geq \delta r_{\min }$.
    If $r_{\psi \phi_{N}}>0,$ set $\phi =\phi_{N}$ and if $r_{\psi \phi_{N}}<0$, set $\phi =\phi_{N}j$ where $j$ is any index with $r_{j}<0$.
    Then $|r_{\phi }|\geq \delta r_{\min }^{2}=C$ and 
    \begin{equation*}
        S_{\sigma \phi }([0,1])\subseteq S_{\psi \phi_{N}}([0,1])\subseteq S_{\sigma }([0,1])\cap S_{\tau }([0,1])
    \end{equation*}
    as required.
\end{proof}

The assumption that $K=[0,1]$ is needed only in the proof of the next lemma.

\begin{lemma}\label{l:bsep}
    Suppose the IFS $\mathcal{S}$ has self-similar set $[0,1]$ and satisfies the weak separation condition.
    For each $\delta >0$, there exists a finite set $\mathcal{E}_{\delta }$ so that for any generation $\alpha >0$ and $\sigma ,\tau \in \Lambda_{\alpha }$, either 
    \begin{equation*}
        m(S_{\sigma }([0,1])\cap S_{\tau }([0,1]))<\delta \alpha \text{ or } S_{\sigma }^{-1}\circ S_{\tau }\in \mathcal{E}_{\delta }.
    \end{equation*}
\end{lemma}

\begin{proof}
    Fix $\delta >0$.
    Choose a net interval $\Delta_{0}$ with the maximum number of neighbours and assume $\Delta_{0}\in \mathcal{F}_{\beta }$.
    Proposition~\ref{l:wsp-l} guarantees this is possible.
    Fix $C=C(\delta )$ as in Lemma~\ref{l:delsep}, define 
    \begin{equation}
        \Gamma =\{T\circ S_{\psi }^{-1}:\psi \in \Sigma^{\ast },|r_{\psi }|\geq C\beta r_{\min }^{2},\text{ }T\in V(\Delta_{0})\},
    \end{equation}
    and put 
    \begin{equation*}
        \mathcal{E}_{\delta }=\{f^{-1}\circ g:f,g\in \Gamma \},
    \end{equation*}%
    which is a finite set since $\Gamma $ is finite.

    Let $\sigma ,\tau \in \Lambda_{\alpha }$ be arbitrary with $m(S_{\sigma}([0,1])\cap S_{\tau }([0,1]))\geq \delta \alpha $.
    Choose $\psi ,\phi $ from the conclusion of Lemma~\ref{l:delsep} where, without loss of generality, $\psi =\sigma$.
    Set $\gamma =|r_{\sigma }r_{\phi }|\beta $.

    \begin{claim}
        The interval $\Delta_{1}=S_{\sigma \phi }(\Delta_{0})$ is a net interval of generation $\gamma $ with $V(\Delta_{0})=V(\Delta_{1})$.
    \end{claim}

    \textsc{Proof (of Claim).}
    Let $\Delta_{0}$ have neighbours generated by $S_{\omega_{1}},...,S_{\omega_{m}}$ with $\omega_{i}\in \Lambda_{\beta }$.
    By definition of $\gamma $, $\{\sigma \phi \omega_{1},\ldots ,\sigma \phi \omega_{m}\}$ are words of generation $\Lambda_{\gamma }$.
    Note that $(\inte\Delta_1)\cap K\neq\emptyset$ and that the endpoints of $\Delta_1$ are of the form $S_{\sigma\phi\zeta}$ where $\zeta\in\Lambda_\beta$ so that $\sigma\phi\zeta\in\Lambda_\gamma$.
    In particular, if $\Delta_1\notin\mathcal{F}_\gamma$, then there exists some $\tau\in\Lambda_\gamma$ such that $S_\tau\notin\{S_{\sigma\phi\omega_1},\ldots,S_{\sigma\phi\omega_m}\}$ and $S_\tau([0,1])\cap(\inte \Delta_1)\neq\emptyset$.
    But then there exists some $\Delta_2\in\mathcal{F}_\gamma$ with $\Delta_2\subseteq\Delta_1\cap S_\tau([0,1])$, where $\Delta_2$ has distinct neighbours generated by $\{\omega_1,\ldots,\omega_m\}\cup\{\tau\}$, contradicting the maximality of $m$.
    Thus $\Delta_{1}=\Delta_{2}$ and $\Delta_{1}\in \mathcal{F}_{\gamma }$ with neighbours generated by the $\sigma \phi \omega_{i}$.
    Moreover, since $r_{\sigma \phi }>0$, we have $T_{\Delta_{1}}=S_{\sigma \phi }\circ T_{\Delta_{0}}$, so that 
    \begin{equation*}
        V(\Delta_{1})=\{T_{\Delta_{1}}^{-1}\circ S_{\sigma \phi \omega_{i}}\}_{i=1}^{m}=\{T_{\Delta_{0}}^{-1}\circ S_{\sigma \phi }^{-1}\circ S_{\sigma \phi }\circ S_{\omega_{i}}\}_{i=1}^{m}=V(\Delta_{0})
    \end{equation*}
    as claimed.

    Now we will show that $S_{\sigma }^{-1}\circ S_{\tau }\in \mathcal{E}_{\delta }$.
    Establishing this will complete the proof.
    Since $K=[0,1]$ and $\Delta_{1}\subseteq S_{\sigma }([0,1])\cap S_{\tau }([0,1])$, the words $\sigma $ and $\tau $ must be the prefixes of $\xi_{1},\xi_{2}\in \Lambda_{\gamma }$ which generate neighbours $T_{1},T_{2}$ of $\Delta_{1}$ respectively.
    Let $\xi_{1}=\sigma \psi_{1}$ and $\xi_{2}=\tau \psi_{2}$.
    Since $\xi_{i}\in \Lambda_{\gamma }$ and $\sigma ,\tau \in \Lambda_{\alpha }$, we have for each $i=1,2,$ 
    \begin{equation}\label{e:psbd}
        |r_{\psi_{i}}|\geq \frac{\gamma }{\alpha }r_{\min }\geq \frac{\alpha|r_{\phi }|\beta }{\alpha }r_{\min }^{2}\geq C\beta r_{\min }^{2}.
    \end{equation}
    But since $T_{\Delta_{1}}^{-1}\circ S_{\xi_{i}}=T_{i}$, we have 
    \begin{align*}
        S_{\sigma }^{-1}\circ S_{\tau }& =S_{\psi_{1}}\circ (S_{\xi_{1}}^{-1}\circ S_{\xi_{2}})\circ S_{\psi_{2}}^{-1} \\
                                       & =S_{\psi_{1}}\circ (T_{1}^{-1}\circ T_{\Delta_{1}}^{-1}\circ T_{\Delta_{1}}\circ T_{2})\circ S_{\psi_{2}}^{-1} \\
                                       & =(T_{1}\circ S_{\psi_{1}}^{-1})^{-1}\circ (T_{2}\circ S_{\psi_{2}}^{-1}),
    \end{align*}%
    and this is an element of $\mathcal{E}_{\delta }$ by (\ref{e:psbd}).
\end{proof}

\begin{theorem}\label{t:wsc-wft}
    Suppose the IFS $\mathcal{S}$ satisfies the weak separation condition and has self-similar set $[0,1]$.
    Then $\mathcal{S}$ has the finite neighbour condition.
\end{theorem}
\begin{proof}
    Assume $\mathcal{S}=\{S_{i}\}_{i=1}^{k}$.
    Set 
    \begin{equation*}
        \delta =r_{\min }\cdot \min (\{|v-S_{i}(u)|:1\leq i\leq k,u,v\in \{0,1\},v\neq S_{i}(u)\}>0
    \end{equation*}
    and let $\mathcal{E}_{\delta }$ be the corresponding finite set as in Lemma~\ref{l:bsep}.
    Put 
    \begin{equation*}
        \mathcal{G}=\left\{ g^{-1}\circ f\circ h:f\in \mathcal{E}_{\delta };g,h\in \{\Id,S_{1},\ldots ,S_{k}\}\right\}
    \end{equation*}
    and again note that $\mathcal{G}$ is a finite set.
    We may now define 
    \begin{align*}
        \epsilon_{1}& :=\min \left\{ m(S_{\phi }([0,1])\cap S_{\psi}([0,1])):|r_{\phi }|,|r_{\psi }|\geq r_{\min }^{2},S_{\phi }((0,1))\cap S_{\psi }((0,1))\neq \emptyset \right\} \\
        \epsilon_{2}& :=\min \left\{ m([0,1]\cap f([0,1])):f\in \mathcal{G}\text{ and }f([0,1])\cap (0,1)\neq \emptyset \right\} .
    \end{align*}
    Fix 
    \begin{equation*}
        0<\epsilon \leq \min \{\epsilon_{1},r_{\min }\epsilon_{2}\}
    \end{equation*}
    and note that $\epsilon\leq r_{\min}$.

    It was shown in Theorem~\ref{t:wft-fs} that $\mathcal{S}$ has the finite neighbour condition if and only if $\mathcal{E}_{\mathcal{S}}((0,1))$ (as defined in (\ref{ESNotation})) is finite.
    We will show that $\mathcal{E}_{\mathcal{S}}((0,1))$ is finite by proving the following claim.

    \begin{claim}
        For any $\alpha >0$ and $\sigma ,\tau \in \Lambda_{\alpha }$ with $S_{\sigma }((0,1))\cap S_{\tau }((0,1))\neq \emptyset$, we have
        \begin{equation*}
            m(S_{\sigma }([0,1])\cap S_{\tau }([0,1]))\geq \epsilon \alpha .
        \end{equation*}
    \end{claim}

    Once the claim is verified, we are done since Lemma~\ref{l:bsep} will imply $\mathcal{E}_{\mathcal{S}}((0,1))$ is contained in the finite set $\mathcal{E}_{\epsilon }$ defined in that lemma.

    \textsc{Proof (of Claim).}
    Assume the claim is false.
    Then there exists some $0<\alpha\leq 1$ and $\sigma ,\tau \in \Lambda_{\alpha }$ such that $S_{\sigma}((0,1))\cap S_{\tau }((0,1))\neq \emptyset$, but 
    \begin{equation}\label{contr}
        m(S_{\sigma }([0,1])\cap S_{\tau }([0,1]))<\epsilon \alpha.
    \end{equation}
    Choose $\alpha $ maximal with this property.
    Observe that the choice of $\epsilon \leq \epsilon_{1}$ ensures $\sigma $ and $\tau $ are both words of length at least two.
    This is because if, say, $\sigma $ had length at most one, then $\alpha \geq r_{\min }$.
    Consequently, $|r_{\sigma }|,|r_{\tau}|\geq r_{\min }^{2}$ and thus the definition of $\epsilon_{1}$ would imply that 
    \begin{equation*}
        m\bigl(S_{\sigma }([0,1])\cap S_{\tau }([0,1])\bigr)\geq \epsilon_{1}\geq \epsilon \alpha,
    \end{equation*}%
    which is false.

    Thus we can let $\alpha^{\prime }=\min \{|r_{\sigma^{-}}|,|r_{\tau^{-}}|\}\geq \alpha $ and obtain prefixes $\sigma^{\prime },\tau^{\prime }$ of $\sigma $ and $\tau $ respectively, with $\sigma^{\prime },\tau^{\prime}\in \Lambda_{\alpha^{\prime }}$.
    Note that $(\sigma^{\prime },\tau^{\prime })$ is one of $(\sigma^{-},\tau )$, $(\sigma ,\tau^{-})$ or $ (\sigma^{-},\tau^{-})$.

    We first show that $m(S_{\sigma^{\prime }}([0,1])\cap S_{\tau^{\prime}}([0,1]))\geq \delta \alpha $.
    For notational simplicity, write
    \begin{align*}
        S_{\sigma }([0,1])& =[a,b],& S_{\tau }([0,1])& =[c,d],\\
        S_{\sigma^{\prime }}([0,1])& =[a^{\prime },b^{\prime }] ,\text{ and }& S_{\tau^{\prime}}([0,1])& =[c^{\prime },d^{\prime }].
    \end{align*}
    Since $\epsilon <r_{\min }$, by swapping the roles of $\sigma $ and $\tau$, if necessary, we may assume 
    \begin{equation*}
        S_{\sigma }([0,1])\cap S_{\tau }([0,1])=[c,b],
    \end{equation*}%
    for otherwise, (without loss of generality) $S_{\sigma }([0,1])\subseteq S_{\sigma }([0,1])\cap S_{\tau }([0,1])$ and then $m(S_{\sigma }([0,1])\cap S_{\tau }([0,1]))\geq r_{\min }\alpha \geq \epsilon \alpha $.

    Since $S_{\sigma^{\prime }}([0,1])\supseteq S_{\sigma }([0,1])$, $b^{\prime}\geq b$ and similarly $c^{\prime }\leq c$.
    Moreover, by the maximality of $\alpha $, we cannot have $b=b^{\prime }$ and $c=c^{\prime }$.

    If $b^{\prime }>b$, then $\sigma^{\prime }=\sigma^{-}$.
    Suppose $b=S_{\sigma }(u)$ and $b^{\prime }=S_{\sigma^{-}}(v)$ where $u,v\in \{0,1\}$, and write $\sigma =\sigma^{-}i$.
    Then 
    \begin{equation*}
        b^{\prime }-b=r_{\sigma }(v-S_{i}(u))=|r_{\sigma }||v-S_{i}(u)|,
    \end{equation*}
    where $v\neq S_{i}(u)$ since $b\neq b^{\prime }$.
    By definition of $\delta$, we have $b^{\prime }-b\geq \delta \alpha $, so that 
    \begin{equation*}
        m\bigl(S_{\sigma^{\prime }}([0,1])\cap S_{\tau^{\prime }}([0,1])\bigr)\geq \delta
        \alpha .
    \end{equation*}
    The case $c^{\prime }<c$ follows similarly.

    This proves that $S_{\sigma^{\prime }}^{-1}\circ S_{\tau^{\prime }}\in\mathcal{E}_{\delta }$.
    Since $S_{\sigma }^{-1}\circ S_{\tau }=g^{-1}\circ S_{\sigma^{\prime }}^{-1}\circ S_{\tau^{\prime }}\circ h$ for some $g,h\in \{\Id,S_{1},\ldots ,S_{k}\}$, we conclude that $S_{\sigma}^{-1}\circ S_{\tau }\in \mathcal{G}$.
    Therefore 
    \begin{equation*}
        m\bigl( S_{\sigma }^{-1}\circ S_{\tau }([0,1])\cap \lbrack 0,1]\bigr) \geq \epsilon_{2}
    \end{equation*}%
    and thus 
    \begin{equation*}
        m\bigl( S_{\sigma }([0,1])\cap S_{\tau }([0,1])\bigr) \geq \left\vert r_{\sigma }\right\vert \epsilon_{2}\geq \epsilon \alpha ,
    \end{equation*}%
    which contradicts our initial assumption (\ref{contr}).
\end{proof}

Combined with the earlier results of the paper we have the following
equivalences.

\begin{corollary}\label{c:wsc-equiv}
    Suppose the IFS $\mathcal{S}$ has self-similar set $[0,1]$. The following
    are equivalent:

    \begin{enumerate}
        \item $\mathcal{S}$ satisfies the weak separation condition;
        \item $\mathcal{S}$ satisfies the finite neighbour condition;
        \item $\mathcal{S}$ satisfies the convex generalized finite type condition;
        \item There exists some $c>0$ such that for any $0<\alpha \leq 1,$ words $\sigma ,\tau \in \Lambda_{\alpha }$ and $z,w\in \{0,1\}$, either $S_{\sigma}(z)=S_{\tau }(w)$ or $\left\vert S_{\sigma }(z)-S_{\tau }(w)\right\vert \geq c\alpha$.
    \end{enumerate}
\end{corollary}

\begin{corollary}
    Suppose the IFS $\mathcal{S}$ has self-similar set $[0,1]$ and commensurate contraction factors.
    If $\mathcal{S}$ satisfies the weak separation condition, then $\mathcal{S}$ satisfies \FTCco{}.
\end{corollary}

In \cite{F15}, Feng obtained this conclusion under the additional assumption of a positive, equicontractive IFS.

\begin{remark}
    It would be interesting to know if it is always true that the weak separation condition implies the finite neighbour condition for IFS in $\R$ (without any additional assumptions).
    We do not even know if an IFS that satisfies the open set condition, where the bounded invariant open set $V$ is not a finite union of intervals, necessarily satisfies the finite neighbour condition.
\end{remark}

\end{document}